\documentclass[a4paper,11pt]{article}
\usepackage[latin1]{inputenc}
\usepackage[english]{babel}
\usepackage{amsmath}
\usepackage{amsfonts}
\usepackage{amssymb}
\usepackage{epsfig}
\usepackage{amsopn}
\usepackage{amsthm}
\usepackage{color}
\usepackage{graphicx}
\usepackage{enumerate}
\usepackage{mathrsfs}
\usepackage{cite}
\newtheorem{theorem}{Theorem}[section]

\newtheorem{lemma}[theorem]{Lemma}
\newtheorem{proposition}[theorem]{Proposition}

\newtheorem*{theorem*}{Theorem}
\newtheorem*{lemma*}{Lemma}
\newtheorem*{remark*}{Remark}
\newtheorem*{definition*}{Definition}
\newtheorem*{proposition*}{Proposition}
\newtheorem*{corollary*}{Corollary}
\numberwithin{equation}{section}
\newcommand{\real}{\mathbb{R}}

\def\e{\varepsilon}        

\newcommand{\RR}{\mathbb{R}}

\def\qed{\,\unskip\kern 6pt \penalty 500
\raise -2pt\hbox{\vrule \vbox to8pt{\hrule width 6pt
\vfill\hrule}\vrule}\par}
\definecolor{darkblue}{rgb}{0.05, .05, .65}
\definecolor{darkgreen}{rgb}{0.1, .65, .1}
\definecolor{darkred}{rgb}{0.8,0,0}
\newcommand{\beqn}{\begin{equation}}
\newcommand{\eeqn}{\end{equation}}
\newcommand{\bear}{\begin{eqnarray}}
\newcommand{\eear}{\end{eqnarray}}
\newcommand{\bean}{\begin{eqnarray*}}
\newcommand{\eean}{\end{eqnarray*}}
\begin{document}

\title{\huge \bf Large time behavior for the fast diffusion equation with critical absorption}

\author{
\Large Said Benachour\,\footnote{Universit\'e de Lorraine, Institut Elie Cartan, UMR 7502, F--54506, Vandoeuvre-les-Nancy Cedex, France,
\textit{e-mail}: said.benachour@univ-lorraine.fr,}
\\[4pt]
\Large Razvan Gabriel Iagar\,\footnote{Departamento de An\'alisis Matem\'atico, Univ. de Valencia, Dr. Moliner 50, 46100, Burjassot (Valencia), Spain,
\textit{e-mail:} razvan.iagar@uv.es}
\footnote{Institute of
Mathematics of the Romanian Academy, P.O. Box 1--764, RO--014700, Bucharest, Romania.}
\\[4pt]
\Large Philippe Lauren\c cot\,\footnote{Institut de Math\'ematiques de Toulouse, UMR~5219, Universit\'e de Toulouse, CNRS, F--31062 Toulouse Cedex 9, France, \textit{e-mail:} laurenco@math.univ-toulouse.fr}\\ [4pt] }

\date{\today}
\maketitle

\begin{abstract}
We study the large time behavior of nonnegative solutions to the Cauchy problem for a fast diffusion equation with critical zero order absorption
$$
\partial_{t}u-\Delta u^m+u^q=0 \quad \quad \hbox{in} \
(0,\infty)\times\real^N\ ,
$$
with $m_c:=(N-2)_{+}/N<m<1$ and $q=m+2/N$. Given an initial condition $u_0$ decaying arbitrarily fast at infinity, we show that the asymptotic behavior of the corresponding solution $u$ is given by a Barenblatt profile with a logarithmic scaling, thereby extending a previous result requiring a specific algebraic lower bound on $u_0$. A by-product of our analysis is the derivation of sharp gradient estimates and a universal lower bound, which have their own interest and hold true for general exponents $q>1$.
\end{abstract}

\vspace{2.0 cm}

\noindent {\bf AMS Subject Classification:} 35B33, 35B40, 35B45,
35K67.

\medskip

\noindent {\bf Keywords:} large time behavior, fast diffusion,
critical absorption, gradient estimates, lower bound.

\newpage

\section{Introduction and main results}\label{sec1}

In this paper, we deal with the large time behavior of a fast
diffusion equation with absorption, in a special case when the
exponent of the absorption term is critical. More precisely, we
consider the following Cauchy problem
\begin{equation}\label{g1}
\partial_{t}u-\Delta u^m+u^q=0 \quad \hbox{in} \
(0,\infty)\times\real^N,
\end{equation}
with initial condition
\begin{equation}\label{g2}
u(0,x)=u_0(x), \quad \ x\in\real^N,
\end{equation}
where $N\ge 1$,
\begin{equation}
u_0\in L^1(\real^N)\cap L^{\infty}(\real^N), \quad u_0\geq 0, \ u_0\not\equiv 0, \label{spirou}
\end{equation}
and the parameters $m$ and $q$ satisfy
\begin{equation}\label{exp}
m_c:=\frac{(N-2)_+}{N}<m<1, \quad q=q_*:=m+\frac{2}{N}.
\end{equation}

Degenerate and singular parabolic equations with absorption such as \eqref{g1} have been the subject of intensive research during the last decades. In \eqref{g1}, the main feature is the
competition between the diffusion $\Delta u^m$ and the
absorption $-u^q$ which turns out to depend heavily on the exponents $m>0$ and $q>0$. More precisely, a critical exponent $q_*=m+2/N$ has been uncovered which separates different dynamics and the large time behavior for non-critical exponents $q\neq q_*$ is now well understood. Indeed, for the
semilinear case $m=1$ and the slow diffusion case $m>1$, it has been shown that, when $q>q_*$, the effect of the absorption is
negligible, and the large time behavior is given by the diffusion
alone, leading to either Gaussian or Barenblatt profiles \cite{GKS, GV, He, KP1, KP2, KU}.

A more interesting case turns out to be the intermediate range of the absorption exponent $q\in (m,q_*)$, where the competition of the two effects is balanced. For $m\ge 1$, the study of this range has led to the discovery of some special self-similar solutions called \emph{very singular solutions} which play an important role in the description of the large time behavior, see \cite{BPT, CQW, GKS, He, KP2, KPV, KU, PT} for instance. This was an important improvement, as the existence of very singular solutions have been later established for many other different equations.

The study of the fast diffusion case $0<m<1$ was performed later,
but restricted to the range of exponents $m_c<m<1$, as the singular
phenomenon of finite time extinction occurs when $m\in (0,m_c)$.
When $m\in (m_c,1)$, the asymptotic behavior has been also
identified for any $q\ne q_*$, $q>1$, and again very singular
solutions play an important role \cite{Leoni, PZ1, PZ2}. Later, also
the limit case $q=m$ and the extinction case when $q<m$ have been
studied \cite{FGV, FV}, although there are still many open problems
in these ranges, as most of the results are valid only in dimension
$N=1$.

In this paper we focus on the critical absorption exponent $q=q_*$
which is the limiting case above which the effect of the absorption
term is negligible in the large time dynamics. That the diffusion is
almost governing the asymptotic behavior is revealed by the fact
that the asymptotic profile is given by the diffusion, but the
scaling is modified as a result of the influence of the absorption
term and additional logarithmic factors come into play. More
precisely, the solutions converge to a Gaussian or Barenblatt type
profile, subject to corrections in $x$ and $u$ of type powers of
$\log t$. The semilinear case $m=1$ and $q=q_*$ is investigated in
\cite{GKS, GV, He} in any space dimension, while the asymptotic
behavior for the slow diffusion case $m>1$ and $q=q_*$ is the
subject of the celebrated paper \cite{GVaz} (and previously
\cite{GP} in dimension $N=1$), where a new dynamical systems
approach, well-known nowadays as the \emph{$S$-theorem}, has been
introduced to deal with small asymptotic non-autonomous
perturbations of autonomous equations. This approach became then
common when dealing with critical exponents, and a survey of it can
be found in the book \cite{GVazBook}. A slightly better asymptotic
estimate has been later obtained in \cite{QL04} for a larger class
of initial data, using the same stability technique.

\bigskip

\noindent \textbf{Main results.} However,  in spite of the general
interest in literature, the problem of studying the asymptotic
behavior for the fast diffusion case $m_c<m<1$ with critical
exponent $q=q_*$ and establishing an analogous result as the one by
Galaktionov and V\'azquez \cite{GVaz} still remains open for a wide
class of non-negative initial data $u_0$, including in particular
compactly supported ones. The main difficulty to be overcome seemed
to be the following: due to the infinite speed of propagation, a
property which contrasts markedly with the range $m>1$, and to the
nonlinearity of the diffusion which is the main difference with the
semilinear case $m=1$, a suitable control of the tail as
$|x|\to\infty$ of $u(t,x)$ is needed for positive times $t>0$. Of
particular importance is the derivation of a sharp lower bound which
allows one to exclude the convergence to zero in the scaling
variables. This difficulty is by-passed in \cite{SW06} by
establishing the required sharp lower bound as soon as the initial
condition $u_0(x)$ behaves as $C |x|^{-l}$ as $|x|\to\infty$ for
some $C>0$ and $l<2/(1-m)$. This assumption clearly excludes a broad
class of ``classical'' initial data, including compactly supported
ones, and our aim in this paper is to get rid of such an assumption.
An intermediate step is to figure out how does the solution $u$ to
the Cauchy problem \eqref{g1}-\eqref{g2} behave as $|x|\to\infty$
for positive times $t>0$ if it starts from a, say, compactly
supported initial condition $u_0$.

We actually provide an answer to this question, in the form of a sharp lower bound for solutions to \eqref{g1}-\eqref{g2}, which is valid \textbf{for any $q>1$}:

\begin{theorem}\label{pr.lb}
Consider an initial condition $u_0$ satisfying \eqref{spirou}, $q>1$, and let $u$ be the corresponding solution to \eqref{g1}-\eqref{g2}. Then
\begin{equation}
u(t,x) \ge \ell_u(t)\ (1+|x|)^{-2/(1-m)}\ , \qquad (t,x)\in
(0,\infty)\times\RR^N\ , \label{lb}
\end{equation}
with
$$
\ell_u(t) := \left[ u^{(m-1)/2}(t,0) + \sqrt{B_0} \left( 1
+\sqrt{(3-m-2q)_+} \|u_0\|_\infty^{(q-1)/2} \sqrt{t} \right)
t^{-1/2} \right]^{2/(m-1)}\ ,
$$
and
\begin{equation}
B_0 := \frac{1-m}{2m(mN-N+2)}>0\ . \label{fantasio}
\end{equation}
\end{theorem}

Note that $\ell_u$ depends on $u$ and converges to zero as $t\to 0$
and $t\to\infty$ since $m<1$, the latter being a consequence of the
decay to zero of $\|u(t)\|_\infty$ as $t\to\infty$, see \eqref{g0}
below. Moreover, note also that the space dependence of this lower
bound is sharp. Indeed, it says: whatever the initial data is,
during the later evolution, the solution to the Cauchy problem has a
spatial decay at infinity slower than the decay of a Barenblatt
self-similar profile, a property which is inherited from the fast
diffusion equation \cite[Theorem~2.4]{HP85}. In particular, let us
point out a curious \emph{jump of the tails}: if $u_0$ is compactly
supported (no tail at all), or decays as $|x|\to\infty$ with a tail
of the form $|x|^{-l}$, $l>2/(1-m)$, then its tail jumps immediately
to a slower decaying one for positive times. This peculiar property
does not seem to have been noticed in \cite{SW06} where it is rather
shown that \eqref{lb} holds true provided $u_0$ does not decay too
fast as $|x|\to\infty$, namely $u_0(x)\sim C |x|^{-l}$ as
$|x|\to\infty$ for some $C>0$ and $l<2/(1-m)$.

The proof of Theorem~\ref{pr.lb} is based on some \emph{sharp
gradient estimates} for well-chosen negative powers of $u$ which
have their own interest and are given in Theorem~\ref{th.g} below.

This universal lower bound allows for a comparison from below of
general solutions with suitable constructed subsolutions. This is
the main technical tool that enables us to establish the asymptotic
behavior of solutions for a very general class of initial data. More
precisely, our main result is:

\begin{theorem}\label{th.asympt}
Consider an initial condition $u_0$ satisfying \eqref{spirou}, $q=q_*=m+2/N$, and assume further that $u_0$ satisfies
\begin{equation}\label{interm5}
u_0(x)\leq K|x|^{-k}\ , \quad x\in\real^N\ , \quad\text{ with }\; k:=N+\frac{mN(mN-N+2)}{2[2-m+mN(1-m)]}\ ,
\end{equation}
for some $K>0$. Let $u$ be the solution to the Cauchy problem \eqref{g1}-\eqref{g2}. Then
\begin{equation}\label{conv.asympt}
\lim\limits_{t\to\infty}(t\log t)^{1/(q-1)}\left|u(t,x)- \frac{1}{(t\log
t)^{1/(q-1)}} \sigma_{A_*}\left(\frac{x}{t^{1/N(q-1)}(\log
t)^{(1-m)/2(q-1)}}\right)\right|=0,
\end{equation}
uniformly in $\real^N$, where
\begin{equation*}
\sigma_A(y)=\left(A+B_0|y|^2\right)^{1/(m-1)}, \quad
B_0=\frac{1-m}{2m(mN-N+2)}, \quad A>0, 
\end{equation*}
and $A_*$ is uniquely determined and given by
$$
A_* := \left( \frac{\displaystyle 2 \int_0^\infty (1+r)^{q_*/(m-1)} r^{(N-2)/2}\ dr}{\displaystyle N \int_0^\infty (1+r)^{1/(m-1)} r^{(N-2)/2}\ dr} \right)^{N(1-m)/(2(Nm+2-N))}\ .
$$
\end{theorem}

\noindent \textbf{Remarks.} (i) We point out that the profile
$\sigma_A$ is the well-known Barenblatt profile from the theory of the standard fast diffusion equation
\begin{equation}
\partial_t \varphi=\Delta \varphi^m  \quad \hbox{in} \
(0,\infty)\times\real^N \label{fde}
\end{equation}
in the supercritical range $m\in (m_c,1)$, see \cite{VazquezSmoothing} for more information.

(ii) As already mentioned, Shi \& Wang prove Theorem~\ref{th.asympt} in \cite{SW06} under more restrictive conditions on the initial data $u_0$. More precisely, they assume the initial condition to satisfy:
$$
\lim\limits_{|x|\to\infty}|x|^l u_0(x)=C>0 \quad {\rm for} \ k \leq
l<\frac{2}{1-m}\ ,
$$
with $k$ defined in \eqref{interm5}, which satisfies $k \in \left( N
, 2/(1-m) \right)$, since $(N-2)_+/N<m<1$. This condition works well
in view of comparison from below with rescaled Barenblatt-type
profiles, but it has the drawback of not allowing some natural
choices of initial data to be considered: in particular, initial
data $u_0$ with compact support, or fast decay at infinity, or even
with the same decay at infinity as the Barenblatt profiles (that is,
with $l=2/(1-m)$ in the condition above) fail to enter the framework
of \cite{SW06}. Our analysis removes the previous condition and
allows us to consider all these ranges of initial data. However, we
will use (and recall when necessary) some of the technical steps and
results in \cite{SW06}, especially those concerning the use of the
general stability technique to show the convergence part of the
proof of Theorem~\ref{th.asympt}.

\bigskip

\noindent \textbf{Organization of the paper.} In
\emph{Section~\ref{sec.grad}}, we prove some sharp gradient
estimates for \eqref{g1}, which are valid for any $q>1$; this result
is new and interesting by itself, and it is stated in
Theorem~\ref{th.g}. We next prove Theorem~\ref{pr.lb}, which turns
out to be a rather simple consequence of Theorem~\ref{th.g}. In
\emph{Section~\ref{sec.comp}}, we construct suitable subsolutions
that can be used for comparison from below, in view of the previous
lower bound. This is the most involved part of the work, from the
technical point of view, since the approach of \cite{SW06} does not
seem to work. Let us emphasize here that our construction relies on
the fact that the solution $u$ to \eqref{g1}-\eqref{g2} enjoys
suitable decay properties after waiting for some time, as a
consequence of Theorem~\ref{pr.lb}. Finally, we prove
Theorem~\ref{th.asympt} in \emph{Section~\ref{sec.final}}, as a
consequence of the previous analysis and of techniques from
\cite{GVaz, GVazBook, SW06}.

\section{Gradient estimates and lower bound}\label{sec.grad}

In this section we consider $m\in (m_c,1)$, $q>1$, and an initial condition $u_0$ satisfying \eqref{spirou}. By \cite[Theorem~2.1]{PZ2} the Cauchy problem \eqref{g1}-\eqref{g2} has a unique non-negative solution $u\in BC((0,\infty)\times \real^N)$ and classical arguments entail that $u\in C([0,\infty);L^1(\real^N))$. In addition $u$ enjoys the same positivity property as the solutions to the fast diffusion equation \eqref{fde}.

\begin{lemma}\label{le.p}
Consider $q>1$ and an initial condition $u_0$ satisfying \eqref{spirou}. Then the corresponding solution $u$ to
\eqref{g1}-\eqref{g2} satisfies $u(t,x)>0$ for all $(t,x)\in
(0,\infty)\times \RR^N$.
\end{lemma}

\begin{proof}
Let $\sigma$ be the solution to the fast diffusion equation
\begin{eqnarray*}
\partial_t \sigma - \Delta \sigma^m & = & 0\ , \quad (t,x)\in (0,\infty)\times \RR^N\ , \\
\sigma(0) & = & u_0\ , \quad x\in\RR^N\ .
\end{eqnarray*}
We set $a:= \|u_0\|_\infty^{q-1}>0$ and
$$
\lambda(t) := e^{-at}\ , \qquad s(t) := \frac{e^{(1-m)at} -
1}{(1-m)a}\ , \qquad \ell(t,x) := \lambda(t) \sigma(s(t),x)
$$
for $(t,x)\in (0,\infty)\times\RR^N$. Introducing the parabolic
operator
$$
\mathcal{L} z := \partial_t z - \Delta z^m + a\ z\ ,
$$
we infer from \eqref{g1}, the non-negativity of $u$, and the
comparison principle that
$$
\mathcal{L} u = \left( \|u_0\|_\infty^{q-1} - u^{q-1} \right)\ u \ge
0 \;\;\text{ in }\;\; (0,\infty)\times\RR^N\ .
$$
Next, for $(t,x)\in (0,\infty)\times\RR^N$,
$$
\mathcal{L}\ell(t,x) = (\lambda'+a \lambda)(t) \sigma(s(t),x) +
\left( \lambda s' - \lambda^m \right)(t) \partial_t\sigma(s(t),x) =
0\ .
$$
Since $u(0,x)=u_0(x) =\sigma(0,x)=\ell(0,x)$ for $x\in\RR^N$, the
comparison principle entails that $u\ge \ell$ in
$(0,\infty)\times\RR^N$. Owing to \cite[Th\'eor\`eme~3]{AB79}, the function $\sigma$
is positive in $(0,\infty)\times\RR^N$ and so are $\ell$ and $u$.
\end{proof}

An immediate consequence of Lemma~\ref{le.p} and classical parabolic regularity is that $u\in C^\infty((0,\infty)\times \real^N)$.

We next turn to estimates on the gradient of solutions to \eqref{g1}-\eqref{g2}.

\begin{theorem}\label{th.g}
Consider an initial condition $u_0$ satisfying \eqref{spirou} and
let $u$ be the corresponding solution to \eqref{g1}-\eqref{g2}. Then
\begin{equation}
\left| \nabla u^{(m-1)/2}(t,x) \right| \le \sqrt{(3-m-2q)_+ B_0}\
\|u_0\|_\infty^{(q-1)/2} + \sqrt{\frac{B_0}{t}} \label{ge}
\end{equation}
for $(t,x) \in (0,\infty)\times\RR^N$, the constant $B_0$ being defined in \eqref{fantasio}. In addition,
\begin{equation}
0 < u(t,x) \le \left( \|u_0\|_\infty^{1-q} + (q-1) t \right)^{-1/(q-1)} \ , \quad (t,x)\in (0,\infty)\times \RR^N\ . \label{gaston}
\end{equation}
\end{theorem}

\begin{proof} The proof of Theorem~\ref{th.g} relies on a modified
Bernstein technique and the nonlinear diffusion is handled as
in \cite{Be81}, see also \cite{Zh11} for positive solutions. We reproduce the proof below for the sake of completeness.

\medskip

\noindent\textbf{Step~1.} We first assume that $u_0\in
W^{1,\infty}(\RR^N)$ and there is $\varepsilon>0$ such that $u_0\ge
\varepsilon$ in $\RR^N$. The comparison principle then provides the
following lower and upper bounds
\begin{equation}
0<\left( \varepsilon^{1-q} + (q-1) t \right)^{-1/(q-1)} \le u(t,x)
\le \left( \|u_0\|_\infty^{1-q} + (q-1) t \right)^{-1/(q-1)} \le
\|u_0\|_\infty \label{g0}
\end{equation}
for $(t,x)\in (0,\infty)\times \RR^N$.

Let $\varphi\in
C^2(0,\infty)$ be a positive and monotone function and set
$u:=\varphi(v)$ and $w:=|\nabla v|^2$. We infer from \eqref{g1} that
\begin{equation}
\partial_t v = \frac{\left( \varphi^m \right)'}{\varphi'}(v)\ \Delta v + \frac{\left( \varphi^m \right)''}{\varphi'}(v)\ w - \frac{\varphi^q}{\varphi'}(v)\ , \quad (t,x)\in (0,\infty)\times \RR^N\ . \label{g3}
\end{equation}
We next recall that
\begin{equation}
\Delta w = 2 \sum_{i,j=1}^N \left( \partial_i \partial_j v \right)^2
+ 2 \nabla v \cdot \nabla \Delta v\ . \label{g4}
\end{equation}
It then follows from \eqref{g3} and \eqref{g4} that
\begin{align*}
\partial_t w & = 2\ \nabla v \cdot \left[ \frac{\left( \varphi^m \right)'}{\varphi'}(v)\ \nabla \Delta v + \left( \frac{\left( \varphi^m \right)'}{\varphi'} \right)'(v)\ \Delta v\ \nabla v  + \frac{\left( \varphi^m \right)''}{\varphi'}(v)\ \nabla w \right] \\
& \ + 2 \left( \frac{\left( \varphi^m \right)''}{\varphi'}
\right)'(v)\ w^2 - 2 \left( \frac{\varphi^q}{\varphi'} \right)'(v)\
w
\end{align*}
or equivalently
\begin{align*}
\partial_t w & = \frac{\left( \varphi^m \right)'}{\varphi'}(v) \left[ \Delta w - 2 \sum_{i,j=1}^N \left( \partial_i \partial_j v \right)^2 \right] + 2 \left( \frac{\left( \varphi^m \right)'}{\varphi'} \right)'(v)\ w\ \Delta v \\
& \ + 2 \frac{\left( \varphi^m \right)''}{\varphi'}(v)\ \nabla v \cdot
\nabla w + 2 \left( \frac{\left( \varphi^m \right)''}{\varphi'}
\right)'(v)\ w^2 - 2 \left( \frac{\varphi^q}{\varphi'} \right)'(v)\
w\ .
\end{align*}
It also reads
\begin{align}
\partial_t w & - \frac{\left( \varphi^m \right)'}{\varphi'}(v) \Delta w - \left[ 2 \frac{\left( \varphi^m \right)''}{\varphi'} + \left( \frac{\left( \varphi^m \right)'}{\varphi'} \right)' \right](v)\ \nabla v \cdot \nabla w \nonumber \\
& + \mathcal{S} - 2 \left( \frac{\left( \varphi^m
\right)''}{\varphi'} \right)'(v)\ w^2 + 2 \left(
\frac{\varphi^q}{\varphi'} \right)'(v)\ w = 0\ , \label{g5}
\end{align}
where
\begin{equation*}
\mathcal{S} := 2 \frac{\left( \varphi^m \right)'}{\varphi'}(v)
\sum_{i,j=1}^N \left( \partial_i \partial_j v \right)^2 + 2 \left(
\frac{\left( \varphi^m \right)'}{\varphi'} \right)'(v) \left[
\frac{1}{2} \nabla v \cdot \nabla w - w \Delta v \right]\ .
\end{equation*}
We now use B\'enilan's trick \cite{Be81} to obtain
\begin{align*}
\mathcal{S} = & 2m\varphi^{m-1}(v) \sum_{i,j=1}^N \left( \partial_i \partial_j v \right)^2 + 2m(m-1) \left( \varphi^{m-2} \varphi' \right)(v) \left[ \sum_{i,j=1}^N \partial_i v\ \partial_j v\ \partial_i \partial_ j v - w \sum_{i=1}^N \partial_i ^2 v \right] \\
= & 2m\varphi^{m-1}(v) \sum_{i=1}^N \left[ \left( \partial_i^2 v \right)^2 + (m-1) \frac{\varphi'}{\varphi}(v) \left( \left( \partial_i v \right)^2 - w \right) \partial_i^2 v \right] \\
+ & 2m\varphi^{m-1}(v) \sum_{i\ne j} \left[ \left( \partial_i \partial_j v \right)^2 + (m-1) \frac{\varphi'}{\varphi}(v)\ \partial_i v\ \partial_j v\ \partial_i \partial_j v \right]. \\
\end{align*}
We further estimate $\mathcal{S}$ as follows
\begin{align*}
\mathcal{S} = & 2m\varphi^{m-1}(v) \sum_{i=1}^N \left[ \partial_i^2 v + \frac{m-1}{2} \frac{\varphi'}{\varphi}(v) \left( \left( \partial_i v \right)^2 - w \right) \right]^2 \\
- & 2m\varphi^{m-1}(v) \sum_{i=1}^N \frac{(m-1)^2}{4} \left( \frac{\varphi'}{\varphi} \right)^2(v)\ \left( \left( \partial_i v \right)^2 - w \right)^2 \\
+ & 2m\varphi^{m-1}(v) \sum_{i\ne j} \left[ \partial_i \partial_j v + \frac{m-1}{2} \frac{\varphi'}{\varphi}(v)\ \partial_i v\ \partial_j v \right]^2 \\
- & 2m\varphi^{m-1}(v) \sum_{i\ne j} \frac{(m-1)^2}{4} \left( \frac{\varphi'}{\varphi} \right)^2(v)\ \left( \partial_i v \right)^2\ \left( \partial_j v \right)^2 \\
\ge & - \frac{m(m-1)^2}{2} \left( \varphi^{m-3} (\varphi')^2
\right)(v) (N-1)\ w^2\ .
\end{align*}
Consequently, inserting the previous lower bound in \eqref{g5}, we
find
\begin{equation}
\mathcal{H} w \le 0\ , \quad (t,x)\in (0,\infty)\times \RR^N\ ,
\label{g6}
\end{equation}
the parabolic operator $\mathcal{H}$ being defined by
\begin{equation*}
\mathcal{H} z := \partial_t z - m \varphi^{m-1}(v) \Delta z - \left[
2 \frac{\left( \varphi^m \right)''}{\varphi'} + \left( \frac{\left(
\varphi^m \right)'}{\varphi'} \right)' \right](v)\ \nabla v \cdot
\nabla z + \mathcal{R}_1(v)\ z^2 + \mathcal{R}_2(v)\ z,
\end{equation*}
with
\begin{eqnarray}
\mathcal{R}_1 & := & - 2 \left( \frac{\left( \varphi^m \right)''}{\varphi'} \right)' - \frac{m(m-1)^2(N-1)}{2} \varphi^{m-3} (\varphi')^2\ , \label{g6a} \\
\mathcal{R}_2 & := & 2 \left( \frac{\varphi^q}{\varphi'} \right)'\ .
\label{g6b}
\end{eqnarray}

We now choose $\varphi(r)= r^{2/(m-1)}$, $r>0$. Then
$$
\left( \frac{\left( \varphi^m \right)''}{\varphi'} \right)(r) =
\frac{m(m+1)}{m-1}\ r\ , \quad \left( \varphi^{m-3} (\varphi')^2
\right)(r) = \frac{4}{(m-1)^2}\ ,
$$
so that
$$
\mathcal{R}_1(v) = \frac{2m}{1-m} \left( mN + 2 - N \right)\ , \quad
\mathcal{R}_2(v) = (2q+m-3)\ v^{2(q-1)/(m-1)}\ .
$$
\noindent Observe that $mN+2-N>0$ due to $m>(N-2)_+/N$ so that
$\mathcal{R}_1(v)>0$.

We next divide the analysis into two cases
depending on the sign of $2q+m-3$.

\medskip

\noindent (a) If $q\ge (3-m)/2$, it follows that
$\mathcal{R}_2(v)\ge 0$. Recalling that the constant $B_0$ is defined in \eqref{fantasio}, the function
$$
W_1(t) := \frac{B_0}{t}\ , \quad t>0\ ,
$$
clearly satisfies
$$
\mathcal{H} W_1 \ge 0 \;\;\text{ in }\;\; (0,\infty)\times \RR^N \;\;\text{  with }\;\; W_1(0)=\infty\ .
$$
 We infer from \eqref{g6} and the comparison principle that
\begin{equation*}
\left| \nabla u^{(m-1)/2}(t,x) \right| \le \sqrt{\frac{B_0}{t}}\ ,\
\quad (t,x) \in (0,\infty)\times \RR^N\ ,
\end{equation*}
recalling that $u^{(m-1)/2}$ is well-defined since $u>0$ by
\eqref{g0}. We have thus proved \eqref{ge} in that case.

\medskip

\noindent (b) In the complementary case $q\in (1, (3-m)/2)$, set $$
A := (3-m-2q) \|u_0\|_\infty^{q-1} B_0>0 \;\;\text{ and }\;\; W_2(t) := A + \frac{B_0}{t}\ , \quad t>0\ ,
$$
the constant $B_0$ being defined in \eqref{fantasio}. We infer from \eqref{g0} and the definition of $v$ that
\begin{align*}
\mathcal{H} W_2 & = - \frac{B_0}{t^2} + \frac{1}{B_0} \left( A + \frac{B_0}{t} \right)^2 - (3-m-2q) \left( A + \frac{B_0}{t} \right)\ u(t,x)^{q-1} \\
& \ge \frac{A^2}{B_0} + \frac{2A}{t} - (3-m-2q) \left( A + \frac{B_0}{t} \right)\ \|u_0\|_\infty^{q-1} \\
& \ge \frac{A}{B_0} \left( A - (3-m-2q) \|u_0\|_\infty^{q-1} B_0 \right) + \frac{2}{t} \left( A - \frac{(3-m-2q) B_0}{2} \|u_0\|_\infty^{q-1} \right) \\
& \ge 0\ .
\end{align*}
Thus
$$
\mathcal{H} W_2 \ge 0 \;\;\text{ in }\;\; (0,\infty)\times \RR^N \;\;\text{ with }W_2(0)=\infty\ .
$$
The comparison principle and \eqref{g6} imply that
$$
w(t,x) \le W_2(t)\ , \quad (t,x) \in (0,\infty)\times \RR^N\ .
$$
Combining this estimate with the subadditivity of the square root gives
\eqref{ge}.

\medskip

\noindent\textbf{Step~2.} We now consider $u_0$ satisfying
\eqref{spirou} and denote the corresponding solution to
\eqref{g1}-\eqref{g2} by $u$. For $\varepsilon>0$, classical
approximation arguments allow us to construct a family of functions
$(u_{0,\varepsilon})_\varepsilon$ such that $\varepsilon <
u_{0,\varepsilon} < \|u_0\|_\infty + 2\varepsilon$,
$u_{0,\varepsilon}\in W^{1,\infty}(\RR^N)$, and
$(u_{0,\varepsilon})_\varepsilon$ converges a.e. in $\RR^N$ towards
$u_0$ as $\varepsilon\to 0$. Denoting the corresponding solution to
\eqref{g1}-\eqref{g2} with initial condition $u_{0,\varepsilon}$ by
$u_\varepsilon$, it follows from Step~1 that $u_\varepsilon$
satisfies \eqref{ge}. Classical stability results guarantee that
$(u_\varepsilon)_\varepsilon$ converges towards $u$ uniformly on
compacts subsets of $(0,\infty)\times \RR^N$ and in
$C([0,\infty);L^1(\real^N))$ as $\varepsilon\to 0$. Since $u>0$ in
$(0,\infty)\times\RR^N$ by Lemma~\ref{le.p}, the validity of the
estimate \eqref{ge} for $u$ is a consequence of the estimate
\eqref{ge} for $u_\varepsilon$ and the upper bound on
$\|u_{0,\varepsilon}\|_\infty$.

\medskip

Finally, the bounds \eqref{gaston} readily follow from Lemma~\ref{le.p} and \eqref{g0}.
\end{proof}

Thanks to the just established gradient estimate, we can improve the positivity statement of Lemma~\ref{le.p} and prove Theorem~\ref{pr.lb}, which is now a simple consequence of Lemma~\ref{le.p}.

\begin{proof}[Proof of Theorem \ref{pr.lb}]
We infer from the positivity of $u$ (see Lemma~\ref{le.p}) and
\eqref{ge} that, for $(t,x)\in (0,\infty)\times\RR^N$,
\begin{align*}
u^{(m-1)/2}(t,x) \le & u^{(m-1)/2}(t,0) + \left\| \nabla u^{(m-1)/2}(t) \right\|_\infty\ |x| \\
\le & u^{(m-1)/2}(t,0) + \sqrt{B_0} \left( \sqrt{(3-m-2q)_+} \|u_0\|_\infty^{(q-1)/2} + t^{-1/2} \right) |x| \\
\le & \ell_u(t)^{(m-1)/2} (1+|x|)\ .
\end{align*}
We thus obtain the estimate \eqref{lb} in Theorem~\ref{pr.lb}, since $m<1$.
\end{proof}

We end up this section by reporting a further consequence of
Theorem~\ref{th.g}, which is a somewhat less precise version of
Theorem~\ref{pr.lb} but will be needed in the sequel.

\begin{proposition}\label{pr.lb2}
Consider $q>1$ and an initial condition $u_0$ satisfying
\eqref{spirou} and let $u$ be the corresponding solution to
\eqref{g1}-\eqref{g2}. Given $\varepsilon\in (0,1)$, there are
$\tau_\varepsilon\ge 1/\varepsilon$ and $\kappa_\varepsilon\ge
1/\varepsilon$ depending on $N$, $m$, $q$, $u_0$, and $\varepsilon$
such that
\begin{equation}
u^{m-1}(\tau_\varepsilon,x) \le \kappa_\varepsilon + \varepsilon |x|^2\ , \quad x\in\real^N\ . \label{lb2}
\end{equation}
\end{proposition}

\begin{proof}
Let $(t,x)\in (0,\infty)\times \real^N$. We infer from \eqref{ge}, \eqref{gaston}, and the positivity of $u$ established in Lemma~\ref{le.p} that
\begin{equation*}
\begin{split}
u^{(m-1)/2}(t,x)&\leq u^{(m-1)/2}(t,0)+\|\nabla
u^{(m-1)/2}(t)\|_{\infty}|x|\\
&\leq
u^{(m-1)/2}(t,0)+C_1\left[\left\|u\left(\frac{t}{2}\right)\right\|^{(q-1)/2}_{\infty}+\sqrt{\frac{2}{t}}\right]|x|\\
&\leq u^{(m-1)/2}(t,0) + C_1 \left[ \sqrt{\frac{2}{(q-1) t}} + \sqrt{\frac{2}{t}} \right] |x|,
\end{split}
\end{equation*}
for some $C_1>0$ depending only on $m$ and $q$, hence
\begin{align*}
u^{m-1}(t,x) & \leq 2u^{m-1}(t,0) + 4C_1^2 \left[ \frac{2}{(q-1)t} +\frac{2}{t} \right]|x|^2\\
& \leq 2u^{m-1}(t,0) + \frac{8 q C_1^2}{(q-1)t}\ |x|^2\ .
\end{align*}
It follows from the previous estimate that there is $t_\varepsilon>1/\varepsilon$ depending only on $N$, $m$, $q$, and $\varepsilon$ such that
$$
u^{m-1}(t,x)\leq 2u^{m-1}(t,0) + \varepsilon |x|^2\ , \quad (t,x)\in (t_\varepsilon,\infty)\times \real^N\ .
$$
Using once more \eqref{gaston} together with $m<1$ gives the existence of $\tau_\varepsilon>t_\varepsilon$ such that $\kappa_\varepsilon := 2u^{m-1}(\tau_\varepsilon,0)> 1/\varepsilon$ and completes the proof.
\end{proof}

\section{Subsolutions and supersolutions}\label{sec.comp}

We restrict our analysis to the critical case $q=q_*$ from now on.
Consider an initial condition $u_0$ satisfying \eqref{spirou} and
let $u$ be the corresponding solution to the Cauchy problem
\eqref{g1}-\eqref{g2}. Fix $T>0$. We perform the change to
self-similar variables
\begin{equation}\label{SSV}
\left\{\begin{array}{l}
v(s,y):=\left[(T+t)\log(T+t)\right]^{1/(q-1)}u(t,x),\\
 \\
\displaystyle y:=\frac{x}{(T+t)^{1/N(q-1)}(\log(T+t))^{(1-m)/2(q-1)}}, \quad s:=\log(T+t),\end{array}\right.
\end{equation}
and notice that \eqref{g1} implies that $v$ solves
\begin{equation}
\partial_{s}v - \mathcal{L} v=0 \;\;\text{ in }\;\; (\log T, \infty)\times \real^N\ , \label{snowwhite}
\end{equation}
with $v(\log T)=u_0$, where $\mathcal{L}$ is the
following nonlinear differential operator:
\begin{equation}\label{selfsim.oper}
\begin{split}
\mathcal{L} z&:=\Delta z^m + \frac{1}{N(q-1)} \left(Nz+y\cdot\nabla z\right)\\
& +\frac{1}{(q-1)s} \left(z+\frac{1-m}{2}y\cdot\nabla
z\right) - \frac{z^q}{s},
\end{split}
\end{equation}
with $q=q_*=m+2/N$.

The aim of this section is to construct subsolutions and supersolutions to \eqref{snowwhite} having the correct time scale and a form similar to the expected asymptotic profile.

\paragraph{Construction of subsolutions.} We recall that the Barenblatt profiles are defined by
\begin{equation}\label{stat.Bar}
\sigma_A(y)=\left(A+B_0|y|^2\right)^{1/(m-1)}, \quad
B_0=\frac{1-m}{2m(Nm-N+2)},
\end{equation}
where $A>0$ is a free parameter (to be chosen later according to our aims) and $B_0>0$, since $m_c<m<1$. With the above notations, we have the following result:

\begin{lemma}\label{lem.sub}
There is $A_{sub}>0$ depending only on $N$ and $m$ such that:
\begin{itemize}
\item[(i)] If $m\in[(N-1)/N,1)$, then
$$
w_{A}(s,y):=\sigma_A(y)\ , \quad (s,y)\in (0,\infty)\times \real^N\ ,
$$
is a subsolution to \eqref{snowwhite} in $(0,\infty)\times \real^N$ for any $A\ge A_{sub}$.

\item[(ii)] If $m_c<m<(N-1)/N$, the function
\begin{equation}\label{subs2}
w_A(s,y):=\sigma_A(y)\left(1-\frac{\gamma}{s}\right), \quad (s,y)\in (0,\infty)\times \real^N\ , \quad
\gamma:=\frac{1}{2(1-m)}>0,
\end{equation}
is a subsolution to \eqref{snowwhite} in $(s_0,\infty)\times \real^N$ for $A\ge A_{sub}$ and
$$
s_0 := \max\left\{\frac{4q}{1-m},\frac{2^{m+2}q}{q-1} \right\} \ .
$$
\end{itemize}
\end{lemma}

\begin{proof}
(i) It is easy to check that
\begin{equation*}
\begin{split}
\Delta\sigma_A^m(y)&=\frac{4B_0m}{(m-1)^2}\frac{B_0|y|^2}{A+B_0|y|^2}\sigma_A(y)+\frac{2NB_0m}{m-1}\sigma_A(y)\\
&=\left[\frac{4B_0m}{(m-1)^2}+\frac{2NB_0m}{m-1}\right]\sigma_A(y)-\frac{4B_0m}{(m-1)^2}\frac{A}{A+B_0|y|^2}\sigma_A(y)\ ,
\end{split}
\end{equation*}
and
$$
\frac{1}{N(q-1)}\left(N\sigma_A(y)+y\cdot\nabla\sigma_A(y)\right)=-\frac{\sigma_A(y)}{1-m}+\frac{2\sigma_A(y)}{(1-m)(mN-N+2)}\frac{A}{A+B_0|y|^2}\ ,
$$
and moreover
$$
\sigma_A(y)+\frac{1-m}{2}y\cdot\nabla\sigma_A(y)=\frac{A}{A+B_0|y|^2}\sigma_A(y)\ .
$$
Consequently, by direct calculation, we find that
\begin{equation*}
\begin{split}
& \frac{1}{\sigma_A(y)} \left( \partial_s \sigma_A - \mathcal{L}\sigma_A \right)(y) \\
&=\frac{1}{1-m}-\frac{2B_0m(mN-N+2)}{(1-m)^2}\\
&\ +\frac{2}{mN-N+2} \left[-\frac{1}{1-m}+\frac{2B_0m(mN-N+2)}{(1-m)^2}\right]
\frac{A}{A+B_0|y|^2}\\
&\ +\frac{1}{(q-1)s}\frac{1}{A+B_0|y|^2}\left[(q-1)\left(A+B_0|y|^2\right)^{(q-1)/(m-1)+1}-A\right]\\
&=\frac{1}{(q-1)s}\frac{1}{A+B_0|y|^2}\left[(q-1)\left(A+B_0|y|^2\right)^{(q-1)/(m-1)+1}-A\right],
\end{split}
\end{equation*}
after noticing that \eqref{stat.Bar} ensures
$$
\frac{1}{1-m}-\frac{2B_0m(mN-N+2)}{(1-m)^2}=0\ .
$$
Since $(N-1)/N\leq m<1$, we remark that
$$
\frac{q-1}{m-1}+1=\frac{2}{m-1}\left(m-\frac{N-1}{N}\right)\leq0,
$$
hence
\begin{equation*}
\begin{split}
\frac{1}{\sigma_A(y)} \left( \partial_s \sigma_A - \mathcal{L}\sigma_A \right)(y) &\leq\frac{1}{(q-1)(A+B_0|y|^2)s}\left[(q-1)A^{(q+m-2)/(m-1)}-A\right]\\
&=\frac{A^{(q+m-2)/(m-1)}}{(q-1)(A+B_0|y|^2)s}\left[(q-1)-A^{(q-1)/(1-m)}\right]\leq0,
\end{split}
\end{equation*}
for $A$ sufficiently large, which ends the proof of (i).

\medskip

\noindent (ii) Let $w_A$ be defined in \eqref{subs2} and set
$\xi=B_0|y|^2$. According to \cite[Proof of Lemma~3.2]{SW06}, we
have, in our notation, that
\begin{equation*}
\begin{split}
\left( \partial_s w_A - \mathcal{L}w_A \right)(s,y)&=\frac{1}{N(q-1)}\left[NA+\frac{N(1-m)-2}{1-m}\xi\right]\frac{\sigma_A(y)}{A+\xi}\left[\left(1-\frac{\gamma}{s}\right)^m-\left(1-\frac{\gamma}{s}\right)\right]\\
&\ +\left(1-\frac{\gamma}{s}\right)^q\frac{\sigma_A(y)^q}{s}-\frac{\sigma_A(y)}{(q-1)s} \left( \frac{A}{A+\xi} \right)\left(1-\frac{\gamma}{s}\right)+\frac{\gamma\sigma_A(y)}{s^2},
\end{split}
\end{equation*}
hence, after some easy rearranging,
\begin{equation}\label{interm1}
\begin{split}
\frac{s}{\sigma_A(y)} \left( \partial_s w_A - \mathcal{L}w_A \right)(s,y)&=\frac{s}{q-1} \left( \frac{A}{A+\xi} \right) \left(1-\frac{\gamma}{s}\right)^m\left[1-\left(1-\frac{\gamma}{s}\right)^{1-m}\right]\\
&\ -\frac{s}{1-m} \left( \frac{\xi}{A+\xi} \right) \left(1-\frac{\gamma}{s}\right)^m\left[1-\left(1-\frac{\gamma}{s}\right)^{1-m}\right]\\
&\ +\left(1-\frac{\gamma}{s}\right)^q\sigma_A(y)^{q-1}+\frac{\gamma}{s}\left[1+\frac{A}{(q-1)(A+\xi)}\right] \\
&\ -\frac{A}{(q-1)(A+\xi)}.
\end{split}
\end{equation}
We next note that
$$
1-\left(1-\frac{\gamma}{s}\right)^{1-m}=(1-m)\int_{-\gamma/s}^0(1+r)^{-m}\,dr,
$$
hence
$$
(1-m)\frac{\gamma}{s}\leq1-\left(1-\frac{\gamma}{s}\right)^{1-m}\leq(1-m)\left(1-\frac{\gamma}{s}\right)^{-m}\frac{\gamma}{s}.
$$
Using the previous inequalities to estimate the first two terms of
\eqref{interm1} and the choice of $\gamma$, we get
\begin{equation}\label{interm2}
\begin{split}
\frac{s}{\sigma_A(y)} \left( \partial_s w_A - \mathcal{L}w_A \right)(s,y) &\leq\frac{(1-m)\gamma}{q-1} \left( \frac{A}{A+\xi} \right) -\gamma\left(1-\frac{\gamma}{s}\right)^m\frac{\xi}{A+\xi}\\
&\ +\sigma_A(y)^{q-1}+\frac{\gamma q}{(q-1)s}-\frac{1}{q-1} \left( \frac{A}{A+\xi} \right)\\
&=(A+\xi)^{(q-1)/(m-1)}+\frac{\gamma q}{(q-1)s} - \frac{1}{2(q-1)} \left( \frac{A}{A+\xi} \right)\\
&\ -\gamma\left(1-\frac{\gamma}{s}\right)^m\frac{\xi}{A+\xi}.
\end{split}
\end{equation}
Since $m\in(m_c,(N-1)/N)$, we notice that
\begin{equation}
0< \frac{q-1}{m-1}+1 = \frac{2-m-q}{1-m}<1\ . \label{asterix}
\end{equation}
Let $R>0$ to be chosen later. We split the analysis into two regions
according to the relative position of $\xi$ and $R$.

\medskip

\noindent \emph{Case 1.} If $\xi\in[0,R]$, then we infer from
\eqref{interm2} that
\begin{equation*}
\begin{split}
\frac{s}{\sigma_A(y)} \left( \partial_s w_A - \mathcal{L}w_A \right)(s,y)&\leq\frac{1}{A+\xi}\left[(A+\xi)^{(2-m-q)/(1-m)}-\frac{A}{2(q-1)} \right]+\frac{\gamma
q}{(q-1)s}\\&\leq\frac{1}{A+\xi}\left[(A+R)^{(2-m-q)/(1-m)}-\frac{A}{2(q-1)} \right]+\frac{\gamma q}{(q-1)s}\ .
\end{split}
\end{equation*}
Taking into account \eqref{asterix}, we realize that, if $A$ is large enough, we can choose $R$ such that
\begin{equation}
(A+R)^{(2-m-q)/(1-m)} \leq \frac{A}{4(q-1)}\ . \label{req1}
\end{equation}
With such a choice of $R$, we deduce
\begin{equation*}
\begin{split}
\frac{s}{\sigma_A(y)} \left( \partial_s w_A - \mathcal{L}w_A \right)(s,y)&\leq-\frac{A}{4(q-1)(A+\xi)}+\frac{\gamma
q}{(q-1)s}\\&\leq\frac{\gamma q}{(q-1)s}-\frac{A}{4(q-1)(A+R)}\leq0,
\end{split}
\end{equation*}
provided
\begin{equation}
s\geq 4q\gamma\frac{A+R}{A}\ . \label{req2}
\end{equation}

\medskip

\noindent \emph{Case 2.} If $\xi\geq R$ and $s\geq2\gamma$, then $(1-\gamma/s)^m \ge 2^{-m}$ and  we infer
from \eqref{interm2} and \eqref{asterix} that
\begin{equation}\label{interm3}
\begin{split}
\frac{s}{\sigma_A(y)} & \left( \partial_s w_A - \mathcal{L}w_A \right)(s,y)\\
&\leq(A+\xi)^{(q-1)/(m-1)}+\frac{\gamma q}{(q-1)s}-\gamma\left(1-\frac{\gamma}{s}\right)^m\frac{\xi}{A+\xi}\\
& \leq \frac{(A+\xi)^{(2-m-q)/(1-m)}}{A+\xi}+\frac{\gamma q}{(q-1)s}- \frac{\gamma}{2^m} \frac{\xi}{A+\xi}\\
&\leq\frac{1}{A+\xi}\left[(A+\xi)^{(2-m-q)/(1-m)}-\frac{\gamma}{2^{m+1}}\xi\right]\\
&\ +\frac{\gamma q}{(q-1)s}-\frac{\gamma\xi}{2^{m+1}(A+\xi)}\\
&\leq\frac{1}{A+\xi}\left[A^{(2-m-q)/(1-m)}+\xi^{(2-m-q)/(1-m)}-\frac{\gamma}{2^{m+1}}\xi\right]\\
&\ +\frac{\gamma q}{(q-1)s}-\frac{\gamma R}{2^{m+1}(A+R)}\\
&\leq\frac{1}{A+\xi}\left[A^{(2-m-q)/(1-m)}+\left(R^{(q-1)/(m-1)}-\frac{\gamma}{2^{m+1}}\right)\xi\right]\\
&\ +\gamma\left[\frac{q}{(q-1)s}-\frac{R}{2^{m+1}(A+R)}\right].
\end{split}
\end{equation}
Choosing now $R>0$ and $s$ such that
\begin{equation}
R^{(q-1)/(m-1)}\leq\frac{\gamma}{2^{m+2}} \quad \hbox{and} \quad
\frac{2^{m+1}q(A+R)}{(q-1)R}\leq s\ , \label{req3}
\end{equation}
we derive from \eqref{interm3} that
\begin{equation*}
\begin{split}
\frac{s}{\sigma_A(y)} \left( \partial_s w_A - \mathcal{L}w_A \right)(s,y)&\leq\frac{1}{A+\xi}\left[A^{(2-m-q)/(1-m)}-\frac{\gamma}{2^{m+2}}\xi\right]\\
&\leq\frac{1}{A+\xi}\left[A^{(2-m-q)/(1-m)}-\frac{\gamma}{2^{m+2}}R\right]\leq 0,
\end{split}
\end{equation*}
if
\begin{equation}
A^{(2-m-q)/(1-m)}\leq2^{-(m+2)}\gamma R\ . \label{req4}
\end{equation}

Gathering the two cases, we have thus shown that
$(\partial_sw_A-\mathcal{L} w_A)(s,y)\le 0$ for $y\in\real^N$
provided the conditions \eqref{req1}, \eqref{req2}, \eqref{req3},
\eqref{req4}, and $s\ge 2\gamma$ are satisfied simultaneously by
$R$, $A$, and $s$. We now let $R=A$, so that these conditions become
$$
(2A)^{(2-m-q)/(1-m)}\leq\frac{A}{4(q-1)}, \quad
A^{(q-1)/(m-1)}\leq\frac{\gamma}{2^{m+2}}
$$
or equivalently
\begin{equation}\label{intermcond1}
A^{(q-1)/(m-1)} \leq \min\left\{ \frac{2^{(m+q-2)/(1-m)}}{4(q-1)} , \frac{\gamma}{2^{m+2}} \right\}\ ,
\end{equation}
and
\begin{equation*}
s\geq s_0 := \max\left\{8\gamma q,2\gamma,\frac{2^{m+2}q}{q-1}\right\}\ .
\end{equation*}
Since $(q-1)/(m-1)<0$, we notice that \eqref{intermcond1} is satisfied provided $A$ is sufficiently large. We have thereby shown that $w_A$ is a subsolution to \eqref{snowwhite} in $(s_0,\infty)\times \real^N$ for $A$ large enough.
\end{proof}

\paragraph{Comparison with subsolutions.} We show now that
the subsolutions constructed above are indeed useful to investigate the large time asymptotics of \eqref{g1}-\eqref{g2}. Let $u$ be the solution to the Cauchy problem \eqref{g1}-\eqref{g2} with initial condition $u_0$ satisfying \eqref{spirou} and exponents $(m,q)$ given by \eqref{exp}. Then the rescaled function $v$ obtained from $u$ via the transformation~\eqref{SSV} enjoys the following property:

\begin{proposition}\label{prop.1}
Let $u_0$ be an initial condition satisfying \eqref{spirou} and
denote the corresponding solution to \eqref{g1}-\eqref{g2} by $u$.
Let $v$ be its rescaled version defined by \eqref{SSV} and consider
$T\ge e^{s_0}$. There are $A_T\ge A_{sub}$, $s_T>0$, and
$\gamma_T>0$ depending only on $N$, $m$, $u_0$, and $T$ such that
\begin{equation}
v(s,y) \ge \left( 1 - \frac{\gamma_T}{s} \right) \left( 1 - \gamma_T e^{-s} \right) w_{A_T}(s,y)\ , \quad (s,y)\in (s_T,\infty)\times \real^N\ , \label{comp.below}
\end{equation}
where $w_{A_T}$ is defined in Lemma~\ref{lem.sub}.
\end{proposition}

\begin{proof}
For $t\ge 1$ we define
$$
a_t:=\left(t\log t\right)^{1/(q-1)}, \quad b_t:=t^{1/N(q-1)}(\log
t)^{(1-m)/2(q-1)},
$$
and
$$
c_t:=\left\{\begin{array}{lcl}
1 & {\rm if} &\displaystyle m\in\left[\frac{N-1}{N},1\right),\\
 & & \\
\displaystyle 1-\frac{1}{2(1-m)\log t} & {\rm if} & \displaystyle m\in\left(\frac{N-2}{N},\frac{N-1}{N}\right).
\end{array}\right.
$$
Fix $\varepsilon_T\in (0,B_0 c_T^{m-1}/T)$ such that $c_T^{1-m} a_T^{m-1} \ge \varepsilon_T A_{sub}$. According to Proposition~\ref{pr.lb2} there are $\tau_T\ge 1/\varepsilon_T$ and $\kappa_T\ge 1/\varepsilon_T$ such that
\begin{equation}
u^{m-1}(\tau_T,x) \le \kappa_T + \varepsilon_T |x|^2\ , \quad x\in \real^N\ . \label{i10}
\end{equation}
Define now the function $V$ by
$$
V(\log(T+t),y) := a_{T+t} u(t +\tau_T, y b_{T+t})\ , \quad (t,y)\in [0,\infty)\times \real^N\ .
$$
Note that $V$ is defined by \eqref{SSV} with $u(\cdot+\tau_T)$ instead of $u$ and thus satisfies
\begin{equation}
\partial_s V - \mathcal{L}V=0, \;\;\text{ in }\;\; (\log T,\infty)\times\real^N, \quad\text{ with }\;\; V(\log T)=u(\tau_T)\ . \label{shiftedsnowwhite}
\end{equation}
Moreover, thanks to \eqref{i10},
$$
V^{m-1}(\log T,y) = a_T^{m-1} u^{m-1}(\tau_T ,y b_T) \le a_T^{m-1} \kappa_T + \varepsilon_T a_T^{m-1} b_T^2 |y|^2\ .
$$
Since $a_T^{m-1} b_T^2 = T$ and $\varepsilon_T c_T^{1-m} T \le B_0$,
we obtain
$$
V^{m-1}(\log T,y) \le c_T^{m-1} \left( c_T^{1-m} a_T^{m-1} \kappa_T + \varepsilon_T c_T^{1-m} T |y|^2 \right) \le c_T^{m-1} \left( c_T^{1-m} a_T^{m-1} \kappa_T + B_0 |y|^2 \right)\ .
$$
Recalling that $w_{A_T}(\log T,y) =c_T \sigma_{A_T}(y)$ and $m<1$,
we end up with
$$
V(\log T,y) \ge w_{A_T}(\log T,y)\ , \quad y\in \real^N\ , \quad\text{ with }\;\; A_T := c_T^{1-m} a_T^{m-1} \kappa_T\ .
$$
Now the properties of $\kappa_T$ and $\varepsilon_T$ ensure that
$A_T\ge c_T^{1-m} a_T^{m-1}/ \varepsilon_T \ge A_{sub}$, so that
$w_{A_T}$ is a subsolution to \eqref{snowwhite} in $(\log
T,\infty)\times \real^N$ by Lemma~\ref{lem.sub}. Taking into account
\eqref{shiftedsnowwhite}, the comparison principle entails that
\begin{equation*}
V(s,y) \ge w_{A_T}(s,y)\ , \quad (s,y)\in [\log(T),\infty)\times \real^N\ . 
\end{equation*}
Equivalently
$$
u(t+\tau_T,x) \ge \frac{c_{T+t}}{a_{T+t}} \left( A_T + B_0 \frac{|x|^2}{b_{T+t}^2} \right)^{1/(m-1)}\ , \quad (t,x)\in [0,\infty)\times \real^N\ .
$$
Recalling that $a_{T+t}^{m-1} b_{T+t}^2 = T+t$ and $m<1$ we realize that
\begin{align*}
u(t+\tau_T,x) & \ge c_{T+t} \frac{b_{T+t}^{2/(1-m)}}{a_{T+t}} \left( A_T b_{T+t}^2 + B_0 |x|^2 \right)^{1/(m-1)} \\
& \ge c_{T+t} (T+t)^{1/(1-m)} \left( A_T b_{T+\tau_T+t}^2 + B_0 |x|^2 \right)^{1/(m-1)} \\
& \ge c_{T+t} \left( \frac{T+t}{T+\tau_T+t} \right)^{1/(1-m)} \frac{b_{T+\tau_T+t}^{2/(1-m)}}{a_{T+\tau_T+t}}  \left( A_T b_{T+\tau_T+t}^2 + B_0 |x|^2 \right)^{1/(m-1)} \\
& \ge c_{T+t} \left( \frac{T+t}{T+\tau_T+t} \right)^{1/(1-m)} \frac{1}{a_{T+\tau_T+t}} w_{A_T}\left( \log(T+\tau_T+t) , \frac{x}{b_{T+\tau_T+t}} \right)\ .
\end{align*}
Since
$$
1 - \frac{2}{(1-m) \log(T+t)} \ge 1 - \frac{\log(T+\tau_T)}{\log(T)}\ \frac{2}{(1-m) \log(T+\tau_T+t)}\ , \quad t\ge 0\ ,
$$
and
$$
\frac{T+t}{T+\tau_T+t} = 1 - \frac{\tau_T}{T+\tau_T+t}\ , \quad t\ge 0\ ,
$$
we end up with
$$
u(t,x) \ge \left( 1 - \frac{\gamma_T}{\log(T+t)} \right) \left( 1 - \frac{\gamma_T}{T+t} \right)^{1/(1-m)} \frac{1}{a_{T+t}} w_{A_T}\left( \log(T+t) , \frac{x}{b_{T+t}} \right)
$$
for $(t,x)\in (\tau_T,\infty)\times \real^N$, where
$$
\gamma_T := \max\left\{ \tau_T , \frac{2 \log(T+\tau_T)}{(1-m)\log T}\right\}.
$$
The inequality \eqref{comp.below} then readily follows after setting $s_T:= \log(T+\tau_T)$ and using \eqref{SSV}.
\end{proof}

\paragraph{Construction of supersolutions.} A class of supersolutions to \eqref{snowwhite} is identified in \cite{SW06}.
Using our notation, we recall in the next result the outcome of the
construction performed in \cite[Lemma~3.2]{SW06}.

\begin{lemma}\label{lem.super}
Define
\begin{equation}\label{super}
z_A(s,y):=\sigma_A(y)\left[1+\frac{1}{s}(A+B_0|y|^2)^{\delta}\right]\ , \quad (s,y)\in (0,\infty)\times \real^N\ ,
\end{equation}
where $A>0$ and $\delta := 1/(1-m) -(k/2)$, the parameter $k$ being
defined in \eqref{interm5} and $\sigma_A$ and $B_0$ in
\eqref{stat.Bar}. There are $s_1>0$ and $A_{sup}>0$ depending only
on $N$ and $m$ such that $z_A$ is a supersolution to
\eqref{snowwhite} in $(s_1,\infty)\times\real^N$ for $A\in
(0,A_{sup})$.
\end{lemma}

The statement given in \cite[Lemma~3.2]{SW06} is somewhat less
precise with respect to the dependence of $s_1$, but a careful
inspection of the proof allows one to check that it does not depend
on $A\in (0,1)$.

\begin{proposition}\label{prop.super}
Let $u_0$ be an initial condition satisfying \eqref{spirou} and
\eqref{interm5} and denote the corresponding solution to
\eqref{g1}-\eqref{g2} by $u$. Let $v$ be its rescaled version
defined by \eqref{SSV}. There exists $T(K)>e^{s_1}$ depending only
on $N$, $m$, and $K$, with $K$ given in \eqref{interm5}, such that,
given $T\ge T(K)$, there is $A_T'\in (0,A_{sup})$ depending only on
$N$, $m$, $u_0$, and $T$ such that
\begin{equation}
v(s,y) \le z_{A_T'}(s,y)\ , \qquad (s,y)\in (\log T, \infty)\times \real^N\ . \label{cinderella}
\end{equation}
\end{proposition}

\begin{proof}
Let $T(K)\ge e^{s_1}$ be such that
\begin{equation}
(2B_0)^{k/2}K \le T^{(k-N)/N(q-1)} (\log T)^{(k(1-m)-2q)/2(q-1)}
\;\;\text{ for all }\;\; T\ge T(K)\ , \label{beta1}
\end{equation}
the existence of $T(K)$ being guaranteed by the inequality $k>N$.
Consider $T\ge T(K)$ and let $A>0$ to be specified later. On the one
hand, if $y\in\real^N$ satisfies $|y|^2 \ge A/B_0$, we deduce from
\eqref{interm5}, \eqref{super}, and \eqref{beta1} that
\begin{align*}
z_A(\log T,y) & \ge \frac{\left( A+B_0 |y|^2 \right)^{-k/2}}{\log T} \ge \frac{(2B_0)^{-k/2}}{\log T} |y|^{-k} \\
& \ge K a_T \left( |y| b_T \right)^{-k} \ge a_T u_0(y b_T) = v(\log
T,y)\ ,
\end{align*}
where
$$
a_T = (T \log T)^{1/(q-1)}\ , \quad b_T = T^{1/N(q-1)} (\log T)^{(1-m)/2(q-1)}\ .
$$
On the other hand, if $y\in\real^N$ satisfies $|y|^2 < A/B_0$, then
$$
v(\log T,y) \le a_T \|u_0\|_\infty \;\;\text{ and }\;\; z_A(\log T,y) \ge \left( A+B_0 |y|^2 \right)^{1/(m-1)} \ge (2A)^{1/(m-1)}\ .
$$
Therefore, if $A\le \left( a_T \|u_0\|_\infty \right)^{m-1}/2$, then
$$
v(\log T,y) \le z_A(\log T,y)\ , \quad y\in B_{\sqrt{A/B_0}}(0)\ .
$$
We have thus shown that, if $T\ge T(K)$ and $A\le \left( a_T
\|u_0\|_\infty \right)^{m-1}/2$, then
$$
v(\log T,y) \le z_A(\log T,y)\ , \quad y\in \real^N\ .
$$
Pick now $A_T'\in (0,A_{sup}) \cap \left( 0 , \left( a_T
\|u_0\|_\infty \right)^{m-1}/2 \right]$. The above analysis
guarantees that $v(\log T) \le z_{A_T'}(\log T)$ in $\real^N$, while
$v$ and $z_{A_T'}$ are a solution and a supersolution to
\eqref{snowwhite}, respectively, by \eqref{snowwhite} and
Lemma~\ref{lem.super}. Applying the comparison principle completes
the proof of Proposition~\ref{prop.super}.
\end{proof}

\section{Convergence}\label{sec.final}

The convergence \eqref{conv.asympt} is now a consequence of the
previous analysis and the stability technique developed in
\cite{GVaz, GVazBook}, the latter having already been used in
\cite{SW06} for \eqref{g1}-\eqref{g2}. We briefly recall it for the
sake of completeness in the Appendix and sketch its application in
our framework below.

We fix an initial condition $u_0$ satisfying \eqref{spirou} and
\eqref{interm5} and $T\ge 1 + e^{s_0} + T(K)$, the parameters $s_0$
and $T(K)$ being defined in Lemma~\ref{lem.sub} and
Proposition~\ref{prop.super}, respectively. We denote the
corresponding solution to \eqref{g1}-\eqref{g2} by $u$ and define
its rescaled version $v$ by \eqref{SSV}. We set $A_1:=A_T\ge
A_{sub}$ and $A_2:=A_T'\in (0,A_{sup})$ where $A_T$ and $A_T'$ are
defined in Proposition~\ref{prop.1} and
Proposition~\ref{prop.super}, respectively, and consider the
complete metric space
\begin{equation}\label{metric} X := \left\{
\vartheta_0\in L^1(\real^N)\ :\ w_{A_1}(\log T, y) \le
\vartheta_0(y) \le z_{A_2}(\log T,y)\ , \;\; y\in\real^N \right\}
\end{equation}
endowed with the distance induced by the $L^1$-norm. Recall that
$w_{A_1}$ and $z_{A_2}$ are defined in Lemma~\ref{lem.sub} and
\eqref{super}, respectively.

Let $\vartheta_0\in X$ and consider the solution $\vartheta$ to \begin{equation}
\partial_s \vartheta - \mathcal{L}\vartheta = 0 \;\;\text{ in }\;\; (\log T,\infty)\times \real^N\ , \quad \vartheta(\log T) = \vartheta_0 \;\;\text{ in }\;\; \real^N\ , \label{theta2}
\end{equation}
where $\mathcal{L}$ is defined in \eqref{selfsim.oper}. Observe that $\vartheta$ is actually given by
\begin{equation}
\vartheta(s,y) = \left( s e^s \right)^{1/(q-1)} u_\vartheta\left(
e^s - T , s^{(1-m)/2(q-1)} e^{s/N(q-1)} y \right) \label{theta3}
\end{equation}
for $(s,y)\in (\log T,\infty)\times \real^N$, where $u_\vartheta$
denotes the unique solution to \eqref{g1} with initial condition
$\vartheta_0$ which exists as $\vartheta_0\in X \subset
L^1(\real^N)\cap L^\infty(\real^N)$. This formula guarantees in
particular the existence and uniqueness of $\vartheta$. Furthermore,
$\vartheta$ enjoys several useful properties which we collect now.
First, since $T\ge \max\{ e^{s_0},e^{s_1}\}$ with $s_1$ defined in Lemma~\ref{lem.super}, we infer from Lemma~\ref{lem.sub},
Lemma~\ref{lem.super}, and the comparison principle that
\begin{equation}
w_{A_1}(\log T,y) \le w_{A_1}(s,y) \le \vartheta(s,y) \le z_{A_2}(s,y) \le z_{A_2}(\log T,y) \label{theta4}
\end{equation}
for $(s,y)\in (\log T,\infty)\times\real^N$. Consequently,
\begin{equation}
\vartheta(s)\in X\ , \qquad s\ge \log T\ . \label{theta5}
\end{equation}
It next follows from \eqref{theta3} and Theorem~\ref{th.g} that
\begin{align*}
\left| \nabla \vartheta^{(m-1)/2}(s,y) \right| & = e^{s/2} \left| \nabla u_\vartheta^{(m-1)/2}\left( e^s-T, s^{(1-m)/2(q-1)} e^{s/N(q-1)} y \right) \right| \\
& \le e^{s/2} \left[ \sqrt{(3-m-2q)_+ B_0} \left\| u_\vartheta\left( \frac{e^s}{2} - T \right) \right\|_\infty^{(q-1)/2} + \sqrt{2 B_0 e^{-s}} \right] \\
& \le e^{s/2} \sqrt{\frac{(3-m-2q)_+ B_0}{(s- \log 2) e^{s - \log 2}}}\ \|\vartheta(s-\log 2)\|_\infty^{(q-1)/2} + \sqrt{2B_0} \ .
\end{align*}
We then use \eqref{theta4} and the boundedness of $z_{A_2}$ to conclude that
\begin{equation}
\left| \nabla \vartheta^{(m-1)/2}(s,y) \right| \le C(T)\ , \qquad (s,y)\in (\log T,\infty)\times \real^N\ , \label{theta6}
\end{equation}
for some positive constant $C(T)$ depending only on $N$, $m$, and $T$. Since
$$
\nabla\vartheta = \frac{2}{m-1} \vartheta^{(3-m)/2} \nabla \vartheta^{(m-1)/2} \;\;\text{ and }\;\; \nabla\vartheta^m = \frac{2m}{m-1} \vartheta^{(m+1)/2} \nabla \vartheta^{(m-1)/2}\ ,
$$
the following bounds are a straightforward consequence of \eqref{theta4}, \eqref{theta6}, and the boundedness of $z_{A_2}$:
\begin{equation}
\left| \nabla\vartheta(s,y) \right| + \left| \nabla \vartheta^m(s,y) \right| \le C(T)\ , \qquad (s,y)\in (\log T,\infty)\times \real^N\ . \label{theta7}
\end{equation}
We then infer from \cite{Gi76, Iv98, PV93} and \eqref{theta7} that, given $R>0$, there are $\zeta>0$ and $C(R,\zeta)>0$ depending only on $N$, $m$, and $T$ such that, for $s_2>s_1\ge \log T$ satisfying $|s_2-s_1|\le \zeta$, there holds:
\begin{equation}
|\vartheta(s_2,y) - \vartheta(s_1,y)| \le C(R,\zeta) \sqrt{|s_2-s_1|}\ , \quad y\in B_R(0)\ .\label{theta8}
\end{equation}
Combining the time continuity of $u$ in $L^1(\mathbb{\real}^N)$ with \eqref{theta4} and \eqref{theta8} gives
\begin{equation}
\vartheta\in C([\log T,\infty);L^1(\real^N))\ . \label{theta9}
\end{equation}
Collecting the information obtained so far on the solutions $\vartheta$ to \eqref{theta2} associated to initial data in $X$ we realize that we are in a position to check the validity of the three assumptions \textbf{(H1)-(H3)} required to apply the stability theory from \cite{GVazBook} which are recalled in the Appendix. In our setting the non-autonomous operator $\mathcal{L}$ is defined in \eqref{selfsim.oper} with the metric space $X$ introduced in \eqref{metric}, its autonomous counterpart being
\begin{equation}\label{selfsim.oper2}
Lz := \Delta z^m+\frac{1}{N(q-1)}(Nz+y\cdot\nabla z)\ .
\end{equation}
The evolution equation
\begin{equation}
\partial_s \Phi - L \Phi = 0 \;\;\text{ in }\;\; (\log T,\infty)\times \real^N\ , \label{theta10}
\end{equation}
is related to the fast diffusion equation \eqref{fde} by a (self-similar) change of variables. The bounds \eqref{theta4}, \eqref{theta6}, \eqref{theta7}, and \eqref{theta8} ensure that both \textbf{(H1)} and \textbf{(H2)} are satisfied, after noticing that
$$
\left| L\vartheta(s,y) - \mathcal{L}\vartheta(s,y) \right| \le \frac{C}{s}\ , \qquad (s,y)\in (\log T,\infty)\times \real^N\ .
$$
As for \textbf{(H3)}, it involves only to the fast diffusion equation \eqref{fde} and its self-similar form \eqref{theta10} and we refer to \cite{GVazBook, VazquezSmoothing} for its proof.

We may thus apply Theorem~\ref{th.stab} below to deduce that the
$\omega$-limit set of any solution $\vartheta$ to \eqref{theta2}
starting from an initial condition in $X$ is a subset of
$$
\Omega := \left\{\sigma_A: w_{A_1}(\log T,y)\leq\sigma_A(y)\leq z_{A_2}(\log T,y), \ y\in\real^N \right\}\ .
$$
Since $\sigma_A$ is strictly decreasing with respect to $A$, we
obtain that there are $0<A_3<A_4$ such that $\sigma_A\in\Omega$ if
and only if $A\in[A_3,A_4]$. The remainder of the proof follows
along the same lines as in \cite[Section~4]{SW06} to which we refer.
We nevertheless mention that the $S$-theorem provides only the
convergence in $L^1(\real^N)$ (which is the topology of $X$) and a
further step is needed to achieve the uniform convergence, see
\cite[Section~5]{GVaz} and \cite[Section~4]{SW06}.

\appendix

\section*{Appendix: The stability theorem}
\setcounter{section}{1}
\setcounter{equation}{0}

We briefly recall here for the reader's convenience the S-theorem
introduced by Galaktionov and V\'azquez in \cite{GVaz, GVazBook} and used in Section~\ref{sec.final} to complete the proof of Theorem~\ref{th.asympt}. As a general framework, consider a non-autonomous evolution equation
\begin{equation}\label{eq.pert}
\partial_s \vartheta = \mathcal{L}\vartheta\ ,
\end{equation}
that can be seen as a \emph{small perturbation} of an autonomous
evolution equation with good asymptotic properties
\begin{equation}\label{eq.unpert}
\partial_s\Phi = L\Phi\ ,
\end{equation}
in the sense described by the three assumptions below.  There is a
complete metric space $(X,d)$ which is positively invariant for both
\eqref{eq.pert} and \eqref{eq.unpert} and:

\begin{itemize}
\item[\textbf{(H1)}] The orbit $\{\vartheta(t)\}_{t>0}$ of a solution
$\vartheta\in C([0,\infty);X)$ to \eqref{eq.pert} is relatively compact in $X$. Moreover, if we let
$$
\vartheta^{\tau}(t):=\vartheta(t+\tau)\ , \quad t\ge 0\ , \quad \tau>0,
$$
then $\{\vartheta^{\tau}\}_{\tau>0}$ is relatively compact in
$L^{\infty}_{{\rm loc}}([0,\infty);X)$.

\item[\textbf{(H2)}] Given a solution $\vartheta\in C([0,\infty);X)$ to \eqref{eq.pert}, assume that there is a sequence of positive times $(t_k)_{k\ge 1}$, $t_{k}\to\infty$ such that
$\vartheta(\cdot+t_k)\to\tilde{\vartheta}$ in $L^{\infty}_{{\rm loc}}([0,\infty);X)$ as $k\to\infty$. Then $\tilde{\vartheta}$ is a solution to \eqref{eq.unpert}.

\item[\textbf{(H3)}] Define the $\omega$-limit set $\Omega$ of \eqref{eq.unpert} in $X$ as the set of $f\in X$ such that there is a solution $\Phi\in C([0,\infty);X)$ to \eqref{eq.unpert} and a sequence of positive times $(t_k)_{k\ge 1}$ such that $t_k\to\infty$ and $\Phi(t_k)\longrightarrow f$ in $X$. Then $\Omega$
is non-empty, compact and uniformly stable, that is: for any $\e>0$, there exists $\delta>0$ such that if $\Phi$ is any solution to \eqref{eq.unpert} with $d(\Phi(0),\Omega)\leq\delta$, then
$d(\Phi(t),\Omega)\leq\e$ for any $t>0$, where $d$ is the distance in the complete metric space $X$.
\end{itemize}

The S-theorem then reads:
\begin{theorem}\label{th.stab}
If \textbf{(H1)-(H3)} above are satisfied, then the $\omega$-limit set of any solution $\vartheta\in C([0,\infty);X)$ to \eqref{eq.pert} is contained in $\Omega$.
\end{theorem}

For a detailed proof we refer the reader to \cite{GVaz, GVazBook}.

\section*{Acknowledgments} R. I. is partially supported by the
Spanish project MTM2012-31103. Part of this work has been done
during visits by R. I. to the Institut de Mathématiques de
Toulouse.



\end{document}